\documentclass[a4paper,11pt]{article}
\usepackage{amsmath,amsthm,amssymb,amsfonts}
\usepackage[svgnames,x11names,table]{xcolor}
\usepackage[utf8]{inputenc}
\usepackage{mathtools}
\usepackage{enumitem}

\usepackage{graphicx}
\usepackage{geometry}
\usepackage{framed}  
\usepackage[ruled,vlined]{algorithm2e}
\usepackage{booktabs}
\usepackage{cite}
\usepackage{hyperref} 
\hypersetup{
        colorlinks=true,   % false: boxed links; true: colored links
        linkcolor=blue,    % color of internal links
        citecolor=red,    % color of links to bibliography
        urlcolor=blue      % color of external links
      }
      
\usepackage{tikz}
\usetikzlibrary{positioning}

\usepackage{bbm}

\usepackage{xfrac}
\usepackage{microtype}
\renewcommand{\leq}{\leqslant}
\renewcommand{\geq}{\geqslant}
\renewcommand{\emptyset}{\varnothing}

\newtheorem{theorem}{Theorem}[section]
\newtheorem{lemma}[theorem]{Lemma}
\newtheorem{proposition}[theorem]{Proposition}
\newtheorem{corollary}[theorem]{Corollary}
\newtheorem{definition}[theorem]{Definition}

\newtheorem{remark}[theorem]{Remark}

\newcommand{\n}[1]{\left\|#1 \right\|}

\renewcommand{\t}{\tau}

\newcommand{\R}{\mathbbm R}

\newcommand{\lr}[1]{\left\langle #1\right\rangle}

\DeclareMathOperator{\prox}{prox}
\DeclareMathOperator{\dom}{dom}

\DeclareMathOperator{\range}{range}

\newcommand{\Hilbert}{\mathcal{H}}
\newcommand{\setto}{\rightrightarrows}

\DeclareMathOperator{\Id}{Id}
\DeclareMathOperator*{\argmin}{arg\,min}
\newcommand{\adj}{\ast}

\newcommand{\bz}{\mathbf{z}}
\newcommand{\bx}{\mathbf{x}}
\newcommand{\by}{\mathbf{y}}
\newcommand{\bv}{\mathbf{v}}
\newcommand{\bu}{\mathbf{u}}

\newcommand{\bA}{\mathbf{A}}
\newcommand{\bB}{\mathbf{B}}

\newcommand{\bF}{\mathbf{F}}

\newcommand{\x}{\bar{x}}
\newcommand{\y}{\bar{y}}
\newcommand{\z}{\bar{z}}

\title{A First-Order Algorithm for Decentralised Min-Max Problems}

\author{Yura Malitsky\thanks{Faculty of Mathematics,
                             University of Vienna,
                             Austria. 
                             E-mail:~\href{href:yurii.malitskyi@univie.ac.at}
                                          {yurii.malitskyi@univie.ac.at}}
       \and
       Matthew K. Tam\thanks{School of Mathematics \& Statistics,
                             The University of Melbourne,
                             Australia.
                             E-mail:~\href{href:matthew.tam@unimelb.edu.au}
                                          {matthew.tam@unimelb.edu.au}}
}

\begin{document}
\maketitle

\begin{abstract}
In this work, we consider a connected network of finitely many agents working
cooperatively to solve a min-max problem with convex-concave structure. We
propose a decentralised first-order algorithm which can be viewed as a
non-trivial combination of two algorithms: PG-EXTRA for decentralised minimisation problems and the forward reflected backward method for (non-distributed) min-max problems. In each iteration of our algorithm, each agent computes the gradient of the smooth component of its local objective function as well as the proximal operator of its nonsmooth component, following by a round of communication with its neighbours. Our analysis shows that the sequence generated by the method converges under standard assumptions with non-decaying stepsize.
\end{abstract}

\paragraph{Keywords.} distributed algorithms, min-max problems, operator splitting.

\section{Introduction}\label{s:intro}
We consider a connected network of $n$ agents working cooperatively to solve a \emph{min-max problem} of the form
\begin{equation}\label{eq:minmax}
 \min_{x\in\mathbbm{R}^p}\max_{y\in\mathbbm{R}^d}\sum_{i=1}^n\Psi_i(x,y)\coloneqq
 f_i(x)+\phi_i(x,y)-g_i(y),
\end{equation}
where solutions to \eqref{eq:minmax} are understood in the sense of \emph{saddle-points}~\cite{rockafellar1970monotone}. 
Here $f_i\colon\mathbbm{R}^p\to(-\infty,+\infty]$ and $g_i\colon\mathbbm{R}^d\to(-\infty,+\infty]$ are proper lsc convex, and $\phi_i\colon\mathbbm{R}^p\times\mathbbm{R}^d\to\mathbbm{R}$ is convex-concave with $L$-Lipschitz continuous gradient. In this work, we propose a distributed decentralised first-order algorithm for solving \eqref{eq:minmax} with the following features: (i)~agent $i$ computes the gradient of $\phi_i$ as well as the proximity operators of $f_i$ and $g_i$, and (ii)~communication between agents is only possible when connected to one another by an edge in the network. 

\subsection{Background and Literature Review}
Decentralised algorithms for solving convex minimisation problems over connected networks is a relatively well-understood topic with a wide range of methods in literature. In broad terms, these  can be divided into several classes such as: subgradient-based methods \cite{nedic2009distributed,yuan2016convergence,duchi2011dual}, ADMM-based methods \cite{shi2014linear,ling2015dlm,makhdoumi2017convergence}, and proximal-gradient-type methods \cite{shi2015proximal,li2019decentralized,alghunaim2020decentralized}.
For the purpose of this work, our discussion will focus on a method from the latter class  known as \emph{PG-EXTRA}~\cite{shi2015proximal}.

Consider a connected network of $n$ agents working cooperatively to solve the minimisation problem
\begin{equation*}\label{eq:min}
\min_{x\in\mathbbm{R}^p}\sum_{i=1}^n\psi_i(x)\coloneqq g_i(x)+h_i(x)
\end{equation*}
with $g_i\colon\mathbbm{R}^p\to(-\infty,+\infty]$ proper lsc convex and
$h_i\colon\mathbbm{R}^p\to\mathbbm{R}$ convex and differentiable with
$L$-Lipschitz gradient. We also assume that the \emph{proximal operator}~\cite[Chapter~24]{bauschke2017convex} of $g_i$, given by
\begin{equation}\label{eq:prox}
\prox_{g_i}(x)=\argmin_{y}\bigl(g_i(y)+\frac{1}{2}\|x-y\|^2\bigr), 
\end{equation}
is computable. Each agent $i$ is supposed to have their own functions
$g_i$ and $h_i$ (which they do not share) and communication between agents is limited to
neighbours as defined by the network topology. 
PG-EXTRA is one of the most popular method that adheres to this setting.
Given a matrix $W=(w_{ij})\in\mathbbm{R}^{n\times n}$ that defines
a communication process in the network (see Definition~\ref{d:mixing matrices}) and a stepsize $\tau>0$, the main iteration of PG-EXTRA takes the form
\begin{equation*}%\label{eq:PG-EXTRA full}
\left\{\begin{aligned}
 u^{k+1}_i &= \sum_{j=1}^nw_{ij}x^{k}_j + u^k_i - \frac12\bigl(x_{i}^{k-1}+\sum_{j=1}^nw_{ij}x^{k-1}_j\bigr) - \tau\bigl(\nabla h_i(x^k_i)-\nabla h_i(x^{k-1}_i)\bigr) \\
 x^{k+1}_i &= \prox_{\tau g_i}(u^{k+1}_i),
\end{aligned}\right.
\end{equation*}
where the variables $u_i$ and $x_i$ are local decision variables of agent $i$, and the coefficient $w_{ij}$ is non-zero only if agents $i$ and $j$ are connected.   Denoting $\bx^k=(x_1^k,\dots,x_n^k)^\top$, $\bu^k=(u_1^k,\dots,u^k_n)^\top$, $g(\bx)=\sum_{i=1}^ng_i(x_i)$ and $h(\bx)=\sum_{i=1}^nh_i(x_i)$, this can be written compactly as
\begin{equation}\label{eq:PG-EXTRA-intro}
\left\{\begin{aligned}
 \bu^{k+1} &= W\bx^{k} + \bu^k - \frac{1}{2}(I+W)\bx^{k-1} - \tau\bigl(\nabla h(\bx^k)-\nabla h(\bx^{k-1})\bigr) \\
 \bx^{k+1} &= \prox_{\tau g}(\bu^{k+1}).
\end{aligned}\right.
\end{equation}
It is tempting to try to directly adapt \eqref{eq:PG-EXTRA-intro} to the setting of
\eqref{eq:minmax} by keeping the same general iteration structure, but
considering it in the product space $\R^p\times \R^d$ with
the proximal operator of $g$ replaced by the proximity operators of $f_i$ and $g_i$, and
the gradient of $h$ replaced by the gradients of~$\phi_i$. However, this approach
does not lead to a convergent algorithm without additional restrictive assumptions on the original problem (see Remark~\ref{r:minmax extra}).

Recent progress in the study of min-max problems in the (non-distributed)
convex-concave setting, motivated in part by machine learning applications,
has given rise to a number of new first-order algorithms. These methods can generally be viewed as being of ``extragradient-type'' \cite{korpelevich1976extragradient,tseng2000modified,csetnek2019shadow,malitsky2020forward} and use the gradient of the objective function's smooth component evaluated at two different points in their update rule.  For the purpose of this work, our discussion will focus on the \emph{forward reflected backward method (FoRB)}~\cite{malitsky2020forward}. 
FoRB can be used to solve min-max problems of the form
\begin{equation}\label{eq:minmax forb}
\min_{x\in\mathbbm{R}^p}\max_{y\in\mathbbm{R}^q}f(x)+\phi(x,y)-g(y),
\end{equation}
where $f,g$ are proper lsc convex and $\phi$ is convex-concave with $L$-Lipschitz gradient. FoRB applied to \eqref{eq:minmax forb} with stepsize $\tau>0$ is given by
\begin{equation}\label{eq:FoRB}
\left\{\begin{aligned}
 x^{k+1} &= \prox_{\tau f}\bigl(x^k-2\tau\nabla_x\phi(x^k,y^k)
 +\tau\nabla_x\phi(x^{k-1},y^{k-1})\bigr)  \\
 y^{k+1} &= \prox_{\tau g}\bigl(y^k+2\tau\nabla_y\phi(x^k,y^k)
  -\tau\nabla_y\phi(x^{k-1},y^{k-1})\bigr).
\end{aligned}\right. 
\end{equation}
Note that the iteration \eqref{eq:FoRB} only needs one gradient evaluation per iteration since the past gradient can be stored in memory and reused. Nevertheless,  \eqref{eq:FoRB} is not suited for decentralised distributed implementation.

\subsection{Contributions and Structure}
In this paper, we develop and analyse decentralised distributed algorithms for min-max problems over arbitrarily connected networks. We seek simple methods with good convergence properties such as: non-decaying stepsizes, minimal communication requirements (\emph{i.e.,} one communication round per iteration), and iterations tailored specifically to the properties of the problem (\emph{e.g.,}~proximal operators are used for ``prox-friendly'' functions, gradient evaluations are used for smooth functions). In the most general setting, our mathematical analysis is performed within the framework of \emph{monotone operator splitting}. However, this is just a tool for the analysis; our final results are stated without this abstraction.  Our main contribution (Algorithm~\ref{a:dist minmax}) can interpreted from two different viewpoints; either as: (i) an extension of PG-EXTRA~\eqref{eq:PDHG} to min-max problems, or (ii) an extension of FoRB~\eqref{eq:FoRB} to distributed min-max problems. In this sense, it is a natural counterpart to PG-EXTRA for min-max problems. 

The remainder of this paper is structured as follows. In Section~\ref{s:PG-EXTRA
  derivation}, we recall the derivation of PG-EXTRA through the lens of
primal-dual methods. The primal-dual perspective serves not only to unify and
simplify presentation of many papers in distributed minimisation, but most
importantly, it will also guide us in the development of our new algorithm for min-max problems. In Section~\ref{s:new pd}, we derive a new primal-dual method, which we term the \emph{primal-dual twice reflected (PDTR)} algorithm. In Sections~\ref{s:dist mono}~and~\ref{s:dist minmax}, we use the PDTR algorithm to derive our main results regarding the development and analysis of a new distributed algorithms for monotone inclusions and min-max problems.

\subsection{Notation and Definitions}
We will use $\Hilbert$ (or $\Hilbert'$)  to denote an abstract real Hilbert space with inner-product $\langle\cdot,\cdot\rangle$ and induced norm $\|\cdot\|$, although we will mostly take $\Hilbert=\mathbbm{R}^p$ or $\mathbbm{R}^d\times\mathbbm{R}^p$ equipped with the dot-product. The product space $\Hilbert^n$ is defined as
$$ \Hilbert^n = \{(x_1,\dots,x_n):x_1,\dots,x_n\in\Hilbert\}. $$
The diagonal subspace of the product space $\Hilbert^n$ is denoted
  $$\mathcal{D}_{\Hilbert^n}\coloneqq \{\bx\in\Hilbert^n\colon x_1=\dots=x_n\},$$
although when the space is clear from context (or not important), we will omit the subscript and just write $\mathcal{D}$.  As a rule of thumb, we will reserve normal letters for elements of $\Hilbert$ (\emph{e.g.,} $x\in\Hilbert$) and bold letters for elements of the product space $\Hilbert^n$ (\emph{e.g.,} $\bx=(x_1,\dots,x_n)\in \Hilbert^n$). In the case where $\Hilbert=\mathbbm{R}^p$, it is convenient to identify $\Hilbert^n=(\mathbbm{R}^p)^n$ with $\mathbbm{R}^{n\times p}$ and represent its elements as
$$ \bx=(x_1,\dots,x_n)^\top = \begin{bmatrix}
x_1^\top \\ \vdots \\ x_n^\top
\end{bmatrix} \in\mathbbm{R}^{n\times p}. $$
In this way, multiplication of $\bx$ with a matrix $W=(w_{ij})\in\mathbbm{R}^{n\times n}$ corresponds to forming linear combinations of the coordinate vectors $x_1,\dots,x_n$. Explicitly, 
$$ (W\bx)_i = \sum_{j=1}^nw_{ij}x_j^\top\quad\forall i=1,\dots,n, $$
where $(W\bx)_i$ denotes the $i$th row of $W\bx$. 

Given a function $f\colon\Hilbert\to(-\infty,+\infty]$, its domain is the set $\dom f\coloneqq\{x\in\Hilbert:f(x)<+\infty\}$. Its \emph{subdifferential} at a point $x_0\in\dom f$ is 
$$ \partial f(x_0) = \{v\in\Hilbert\colon f(x_0)+\langle v,x-x_0\rangle \leq f(x)\,\forall x\in\Hilbert\}, $$
and $\partial f(x_0)=\emptyset$ when $x_0\not\in\dom f$. Given a set $C\subset\Hilbert$, its \emph{indicator function} is the function defined by $\iota_C(x)=0$ if $x\in C$ and $\iota_C(x)=+\infty$ otherwise. The \emph{normal cone} to $C$ is given by $N_C=\partial\iota_C$.

We write $A\colon\Hilbert\setto\Hilbert$ to  indicate that $A$ is a set-valued operator, that is, a mapping from points in $\Hilbert$ to subsets of $\Hilbert$. The set-valued inverse of $A$ is denoted $A^{-1}(y)=\{x:y\in A(x)\}$. An operator $A\colon\Hilbert\setto\Hilbert$ is said to be \emph{monotone} if 
$$ \langle x-u,y-v\rangle \geq 0 \quad\forall y\in A(x),v\in A(u). $$
Furthermore, a monotone operator is said to be \emph{maximally monotone} if
there is no monotone operator that properly extends it. Important classes of maximally monotone operators are  subdifferentials of proper lsc convex functions and skew-symmetric matrices. For further details, see \cite[Chapter~20.2]{bauschke2017convex}.

\begin{table*}[t]
\centering\small
\caption{Summary of notation used in the paper.}\label{t:notation}
\begin{tabular}{lccl}\toprule
\textbf{Problem} & \textbf{Vector space} & \qquad &\textbf{Variables}\\\midrule
Minimisation    & $\R^p$ & &$x\in \R^p$  \\
Saddle point & $\R^p\times \R^d$ & &$x\in \R^p$, $y\in \R^d$, $z\in \R^p\times \R^d$  \\
Operator inclusion & $\Hilbert$ & &$x\in \Hilbert$  \\
Primal-dual inclusion  & $\Hilbert\times \Hilbert'$ & &$x\in \Hilbert$,
                                                      $y\in\Hilbert'$, $z\in
                                                      \Hilbert\times\Hilbert'$\\
Product space inclusion  & $\Hilbert^n$& &$\bx=(x_1,\dots,x_n) \in \Hilbert^n$\\
Product space saddle point  & $\R^{n\times p}\times \R^{n\times d}$
                                         & & $\bx=(x_1,\dots,x_n)^\top \in \R^{n\times p}$, $\by=(y_1,\dots,y_n)^\top \in  \R^{n\times d} $ \\
  
  \bottomrule
\end{tabular}
\end{table*}

The \emph{resolvent} of a set-valued operator $A\colon\Hilbert\setto\Hilbert$, is the operator defined by 
 $$ J_A = (I+A)^{-1}. $$
When $A$ is maximally monotone, its resolvent is a well-defined (\emph{i.e.,} single-valued) operator with full domain; see \cite[Chapter~23]{bauschke2017convex}. Important examples of the resolvent for this work include: if $A(x)=\partial f(x)$, then $J_A(x)=\prox_f(x)$ as defined in \eqref{eq:prox}, and (ii) if $A(x,y)=\partial f(x)\times\partial g(y)$, then  $$ J_A(x,y)=\bigl(\prox_f(x),\prox_g(y)\bigr). $$

Before proceeding, we wish to make a general comment on notation. In this paper, we work over many layers of problem formulations and it is hard to keep the notation perfectly consistent across various types of formulations. We will mention minimisation problems, saddle point problems, monotone inclusions, primal-dual monotone inclusions, and finally monotone inclusions in the product space. Thus, to ease the notational burden for the reader, Table~\ref{t:notation} collects the most used notation for each of the problem classes.

\section{PG-EXTRA as a primal-dual algorithm}\label{s:PG-EXTRA derivation}
In order to motivate our approach, we first recall the derivation of PG-EXTRA from the perspective of a primal-dual algorithm \cite{ryu2022large,wu2017decentralized,li2021new}. Since communication in PG-EXTRA is described in terms of \emph{mixing matrices}, we start with the following definition.

\begin{definition}[Mixing matrix]\label{d:mixing matrices}
Consider a simple undirected connected graph $\mathcal{G}=(\mathcal{V},\mathcal{E})$ with
vertices $\mathcal{V}=\{1,\dots,n\}$ and edges $\mathcal{E}$. A
matrix $W\in\mathbbm{R}^{n\times n}$ is called a \emph{mixing matrix}, if it satisfies
\begin{enumerate}[label=(\alph*)]
\item \textsc{(Decentralised property)}\label{d:mm a} $w_{ij}=0$, if $(i,j)\not\in \mathcal{E}$.
\item \textsc{(Symmetry property)}\label{d:mm b} $W=W^\top$.
\item \textsc{(Kernel property)}\label{d:mm c} $\ker(\Id-W)=\mathcal{D}$.
\item \textsc{(Spectral property)}\label{d:mm d} $\Id\succeq W\succ-\Id$.
\end{enumerate}
\end{definition}
Let $\mathcal{L}$ denote the Laplacian matrix of $\mathcal{G}$ \cite[Section~1.1]{brouwer2011spectra} and
$\alpha>\frac{1}{2}\lambda_{\max}(\mathcal{L})$. Then one possible example of a mixing
matrix is given by setting $W=I-\mathcal{L}/\alpha$. Another example is when $W$ is a
symmetric bistochastic matrix. For further details and examples, see \cite[Section~2.4]{shi2015extra}. Regarding the role of the four conditions in Definition~\ref{d:mixing matrices}, conditions~\ref{d:mm b}--\ref{d:mm d} are required for convergence of the algorithm, whereas~\ref{d:mm a} is only required in order to describe decentralised communication over the graph $G$.

Returning to our derivation of PG-EXTRA, we start by considering the minimisation problem
\begin{equation}\label{eq:min g+h}
\min_{x\in\mathbbm{R}^p}\sum_{i=1}^n\left(g_i(x)+h_i(x)\right).
\end{equation}
Denote $\bx = (x_1,\dots,x_n)^\top$, $g(\bx)=\sum_{i=1}^ng_i(x_i)$, and $h(\bx)=\sum_{i=1}^nh_i(x_i)$. Let $W$
be a mixing matrix and set $K=\bigl(\frac{\Id-W}{2}\bigr)^{\sfrac 12}\succeq0$. Since $\ker K=\mathcal{D}$, solving \eqref{eq:min g+h} is equivalent to solving
\begin{equation}\label{eq:ext minimisation}
 \min_{\bx\in \mathbbm{R}^{p\times n}}g(\bx)+h(\bx)+\iota_{\{0\}}(K\bx),
\end{equation} 
and the structure of problem~\eqref{eq:ext minimisation} matches the required assumptions for the \emph{Condat--V\~u primal-dual algorithm}~\cite{condat2013primal,vu2013splitting}. Applying this method to \eqref{eq:ext minimisation} yields the iteration
\begin{equation}\label{eq:cv1} \left\{\begin{aligned}
  \bx^{k+1} &= \prox_{\tau g}\bigl(\bx^k-\tau K^\top\by^k-\tau\nabla h(\bx^k)\bigr) \\
  \by^{k+1} &= \by^k+\frac{1}{\tau} K(2\bx^{k+1}-\bx^k).
 \end{aligned}\right. 
\end{equation}
Assuming the stepsize $\tau>0$ satisfies $\frac{1}{2}\tau L+\|K^2\|<1\iff\tau L<1+\lambda_{\min}(W)$,  the sequence $(\bx^k)$ converges to a point $\bx$ that solves \eqref{eq:ext minimisation}. Note that iteration \eqref{eq:cv1} is not immediately suitable for decentralised communication as the matrix $K$ need not be a mixing matrix.

To eliminate the matrix $K$ from \eqref{eq:cv1}, we define a new sequence $(\bu^k)$ according to
 $$\bu^{k+1}=\bx^k-\tau K\by^k-\tau\nabla h(\bx^k). $$
This sequences satisfies the recursion
\begin{equation*}
\begin{aligned}
 \bu^{k+1} -\bu^k + \tau(\nabla h(\bx^k)-\nabla h(\bx^{k-1}))
  &= \bx^k  -\bx^{k-1} -\tau K(\by^k-\by^{k-1}) \\
  &= \bx^k  -\bx^{k-1} -\frac{1}{2}(\Id-W)(2\bx^k-\bx^{k-1})  \\
  &= W\bx^k-\frac{1}{2}(\Id+W)\bx^{k-1}.
\end{aligned}
\end{equation*}
Thus, altogether, we can rewrite \eqref{eq:cv1} as
\begin{equation}\label{eq:PG-EXTRA}
\left\{\begin{aligned}
\bu^{k+1} &= W\bx^k+\bu^k-\frac{1}{2}(\Id+W)\bx^{k-1} 
 -\tau(\nabla h(\bx^k)-\nabla h(\bx^{k-1})) \\
\bx^{k+1} &= \prox_{\tau g}(\bu^{k+1}), \\
\end{aligned}\right.
\end{equation}
which is precisely PG-EXTRA. As this iteration now involves the mixing matrix $W$, it is suitable for decentralised implementation.

Let us attempt an analogous approach for the min-max problem given by
\begin{equation}\label{eq:minmax2}
\min_{x\in\mathbbm{R}^p}\max_{y\in\mathbbm{R}^d}\sum_{i=1}^n\bigl(f_i(x)+\phi_i(x,y)-g_i(y)\bigr), 
\end{equation}
where $f_i\colon\mathbbm{R}^p\to(-\infty,+\infty]$ and $g_i\colon\mathbbm{R}^d\to(-\infty,+\infty]$ are proper lsc convex, and $\phi_i\colon\mathbbm{R}^p\times\mathbbm{R}^d\to\mathbbm{R}$ is convex-concave with $L$-Lipschitz continuous gradient. 
To do so, denote $g(\bx)=\sum_{i=1}^ng_i(x_i)$, $\phi(\bx,\by)=\sum_{i=1}^n\phi_i(x_i,y_i)$ and $f(\bx)=\sum_{i=1}^nf_i(x_i)$, and suppose $K_1,K_2\in\mathbbm{R}^{n\times n}$ are matrices satisfying $\ker K_1=\ker K_2=\mathcal{D}$. 
Then solving~\eqref{eq:minmax2} is equivalent to solving
\begin{equation}\label{eq:minmax3}
\min_{\bx \in \mathbbm{R}^{n\times p}}\max_{\by\in \mathbbm{R}^{n\times d}}f(\bx)+\phi(\bx,\by)-g(\by)+\iota_{\{0\}}(K_1\bx) - \iota_{\{0\}}(K_2\by),
\end{equation}
which is the min-max analogue of \eqref{eq:ext minimisation}. However, unlike in
the case of \eqref{eq:ext minimisation} where the Condat--V\~u method could be
used, the existing algorithms for solving \eqref{eq:minmax3} often require
additional assumptions
\cite{mukherjee2020decentralized,liu2019decentralized,rogozin2021decentralized}. 
 Indeed, \cite{mukherjee2020decentralized} assumes that $\nabla\phi$ satisfies a ``reverse Lipschitz inequality'' (see \cite[Assumption~3]{mukherjee2020decentralized}),  \cite{kovalev2022optimal} assumes that $\phi$ is strongly convex-strongly concave, and \cite{liu2019decentralized,rogozin2021decentralized,beznosikov2020distributed} require that $f$ and $g$ are indicator functions to compact constraint sets.
Thus, in order to derive a decentralised algorithm for \eqref{eq:minmax3}, we
will first investigate suitable variants of the Condat--V\~u method for~\eqref{eq:minmax3}. For this, it will be more convenient to work with the
abstraction of \emph{monotone inclusions} rather than the min-max problem
directly. Note that, by denoting $\bz=(\bx,\by)$,
the solutions to problem~\eqref{eq:minmax3} can be characterised by
the inclusion
$$  0\in A(\bz)+B(\bz)+K^\ast C(K\bz), $$
where $K(\bz)=(K_1\bx,K_2\by)$ and
$$  A(\bz) = (\partial f(\bx),\partial g(\by)),\quad B(\bz)=(\nabla_x\phi(\bz),-\nabla_y\phi(\bz)),\quad C(\bz)=(N_{\{0\}}(\bx),N_{\{0\}}(\by)). $$
This observation motivates our study of inclusions of the form specified in the following section.

\begin{remark}\label{r:minmax extra}
As stated in Section~\ref{s:intro}, it is tempting to try to directly adapt PG-EXTRA to the setting of \eqref{eq:minmax2} by keeping the same general iteration structure. However, this approach does not lead to a convergent algorithm in this setting. To see this, consider the bilinear min-max problem 
\begin{equation*}
\min_{x\in\mathbbm{R}^p}\max_{y\in\mathbbm{R}^p}\phi(x,y)=\langle x,y\rangle, 
\end{equation*}
which corresponds to \eqref{eq:minmax2} with $n=1$ and $f=g=0$. Note that, for the singleton graph, $W=\begin{bmatrix} 1\end{bmatrix}\in\mathbb{R}^{1\times 1}$ is a mixing matrix. Thus, by setting $z^k=(x^k,y^k)$, the direct adaptation of \eqref{eq:PG-EXTRA} can be expressed as
$$ z^{k+1} = 2z^k - z^{k-1} -\tau\big(B(z^k)-B(z^{k-1})\bigr), $$
where $B(z^k)=(y^k,-x^k)$ and $z^1=z^0-\tau B(z^0)$. Then
\begin{align*}
\|z^{k+1}-z^k\|^2 
 &= \|z^k - z^{k-1} -\tau B(z^k-z^{k-1})\|^2 \\
 &= \|z^k-z^{k-1}\|^2 -2\tau\langle  z^k - z^{k-1},B(z^k-z^{k-1})\rangle + \tau^2\|B(z^k-z^{k-1})\|^2 \\
 &= (1+\tau^2)\|z^k-z^{k-1}\|^2 = \dots = (1+\tau^2)^k\|z^1-z^0\|^2 = (1+\tau^2)^k\|B(z^0)\|^2,
\end{align*}
from which we conclude that $(z^k)$ diverges whenever $z^0\neq 0$.
\end{remark}

\section{A New Primal-Dual Algorithm}\label{s:new pd}
Let $\Hilbert,\Hilbert'$ be real Hilbert spaces, which we assume for simplicity to be
finite-dimensional, with inner-product $\langle\cdot,\cdot\rangle$ and induced
norm $\|\cdot\|$. Since we work in abstract spaces, it is
convenient to use variables $x\in \Hilbert$ and $y\in\Hilbert'$, and
reserve bold fonts for the product space formulations that will appear later in
the text.

Consider the (primal) monotone inclusion
\begin{equation}\label{eq:mono incl}
  0\in (A+B)(x)+K^\ast C(Kx) \subset \Hilbert,
\end{equation}
where $A\colon\Hilbert\setto\Hilbert$ and $C\colon\Hilbert'\setto\Hilbert'$ are maximally monotone, $B\colon\Hilbert\to\Hilbert$ is monotone and $L$-Lipschitz continuous, and $K\colon\Hilbert\to\Hilbert'$ is a linear operator with adjoint $K^\adj$. Associated with this primal inclusion \eqref{eq:mono incl} is the dual inclusion in the sense of \emph{Attouch--Théra}~\cite{attouch1996general} which takes the form
\begin{equation}\label{eq:dual inlc}
 0 \in -K(A+B)^{-1}(-K^\adj y) + C^{-1}(y) \subset \Hilbert'.
\end{equation}
As a consequence of \cite[Corollary~3.2]{attouch1996general}, the primal inclusion~\eqref{eq:mono incl} has a solution if and only if the dual inclusion~\eqref{eq:dual inlc} has a solution. In this case, it is convenient to represent both inclusions together as the primal-dual system\footnotemark{}
\begin{equation}\label{eq:pd incl}
\binom{0}{0} \in \begin{bmatrix}
A+B & K^\adj \\
-K  & C^{-1} \\
\end{bmatrix} \binom{x}{y} \subset\Hilbert\times\Hilbert'. 
\end{equation}

\footnotetext{This follows by combining the following  three  observations:
(i)~If $x\in\Hilbert$ satisfies \eqref{eq:mono incl}, then there exists $y\in C(Kx)$ such that $0\in (A+B)(x) +K^*y$, and hence $\binom{x}{y}$ satisfies \eqref{eq:pd incl}.
(ii)~Similarly, if $y\in\Hilbert'$ satisfies~\eqref{eq:dual inlc}, then there exists $x\in (A+B)^{-1}(-K^*y)$ such that $0\in Kx+C^{-1}(y)$, and hence $\binom{x}{y}$ satisfies \eqref{eq:pd incl}.
(iii)~If $\binom{x}{y}\in\Hilbert\times\Hilbert'$ satisfies \eqref{eq:pd incl}, then $x\in\Hilbert$ satisfies \eqref{eq:mono incl} and $y\in\Hilbert'$ satisfies \eqref{eq:dual inlc}.}

We propose the following primal-dual algorithm which solves \eqref{eq:mono incl} and \eqref{eq:dual inlc} simultaneously. Given initial points $x^{-1}=x^0\in\Hilbert$ and $y^0\in\Hilbert'$, our scheme is given by 
\begin{equation}\label{eq:new algo}
\left\{\begin{aligned}
 x^{k+1} &= J_{\tau A}\bigl(x^k-\tau K^\adj y^k-2\tau B(x^k)+\tau B(x^{k-1})\bigr) \\
 y^{k+1} &= J_{\sigma C^{-1}}\bigl(y^k+\sigma K(2x^{k+1}-x^k)\bigr),
\end{aligned}\right.
\end{equation}
where the stepsizes $\tau,\sigma>0$ are assumed to satisfy $2\tau L+\tau\sigma\|K\|^2<1$.
We refer to the scheme~\eqref{eq:new algo} as the \emph{primal-dual twice reflected (PDTR)} method. Before turning our attention to the analysis of the PDTR, some comments are in order.
\begin{remark}%\label{r:pd1}
In what follows, we consider three special cases in which the proposed algorithm
recovers known methods as well as pointing out closely related methods.
\begin{enumerate}[label=(\alph*)]
\item If $B=0$, then \eqref{eq:new algo} recovers  the \emph{primal-dual hybrid gradient (PDHG) algorithm}~\cite{chambolle2011first}:
\begin{equation*}
\left\{\begin{aligned}
 x^{k+1} &= J_{\tau A}\bigl(x^k-\tau K^\adj y^k\bigr) \\
 y^{k+1} &= J_{\sigma C^{-1}}\bigl(y^k+\sigma K(2x^{k+1}-x^k)\bigr).
\end{aligned}\right.
\end{equation*}
\item If $K=0$, then \eqref{eq:new algo} reduces to two independent iterations:  the \emph{forward reflected backward method}~\cite{malitsky2020forward} for $A+B$ and the \emph{proximal point algorithm} for $C^{-1}$, that is
\begin{equation*}
\left\{\begin{aligned}
 x^{k+1} &= J_{\tau A}\bigl(x^k-2\tau B(x^k)+\tau B(x^{k-1})\bigr) \\
 y^{k+1} &= J_{\sigma C^{-1}}\bigl(y^k\bigr).
\end{aligned}\right.
\end{equation*}
\item \label{r:pd1:frdr} If $K=\Id$, then \eqref{eq:new algo} is equivalent to the \emph{forward--reflected Douglas--Rachford method (FRDR)}~\cite[Section~5]{ryu2020finding} after a change of variables. Indeed, using the identity
  $\Id - J_{\sigma^{-1}C} = \sigma^{-1}J_{\sigma C^{-1}}\cdot(\sigma\Id)$ \cite[Proposition~23.20]{bauschke2017convex} in \eqref{eq:new algo} gives
\begin{equation*}
\left\{\begin{aligned}
 x^{k+1} &= J_{\tau A}\bigl(x^k-\tau y^k-2\tau B(x^k)+\tau B(x^{k-1})\bigr) \\
 y^{k+1} &= J_{\sigma C^{-1}}\bigl(y^k+\sigma (2x^{k+1}-x^k)\bigr) \\
         &= y^k+\sigma(2x^{k+1}-x^k) - \sigma J_{\sigma^{-1}C}\bigl(\sigma^{-1}y^k+2x^{k+1}-x^k)\bigr).
\end{aligned}\right.
\end{equation*}
By setting $\gamma=\sigma^{-1}$ and introducing an extra variable, this can be expressed as
\begin{equation*}
\left\{\begin{aligned}
 x^{k+1} &= J_{\tau A}\bigl(x^k-\tau y^k-2\tau B(x^k)+\tau B(x^{k-1})\bigr) \\
 u^{k+1} &= J_{\gamma C}(2x^{k+1}-x^k+\gamma y^k)\\
 y^{k+1} &= y^k+\frac{1}{\gamma}(2x^{k+1}-x^k-u^{k+1}),
\end{aligned}\right.
\end{equation*}
which is precisely FRDR. Furthermore, since 
$$ 2\lambda L + \tau\sigma\|K\|^2 < 1 \iff \tau < \frac{\gamma}{1+2\gamma L}, $$
we also recover the stepsize for the method as it appears in \cite[Theorem~5.1]{ryu2020finding}.

\item The iteration~\eqref{eq:new algo} is closely related to~\emph{Condat--Vũ splitting} \cite{condat2013primal,vu2013splitting} which takes the form
\begin{equation}\label{eq:PDHG}
\left\{\begin{aligned}
 x^{k+1} &= J_{\tau A}\bigl(x^k-\tau K^\adj y^k-\tau B(x^k)\bigr) \\
 y^{k+1} &= J_{\sigma C^{-1}}\bigl(y^k+\sigma K(2x^{k+1}-x^k)\bigr).
\end{aligned}\right.
\end{equation}
Comparing \eqref{eq:PDHG} with \eqref{eq:new algo}, we observe the main difference  to be the presence of the ``reflection term'' involving the operator $B$. As we shall soon see, the reflection term allows us to establish convergence without assuming cocoercivity of $B$ as is needed for~\eqref{eq:PDHG}.
\end{enumerate}
\end{remark}

We now return to the convergence analysis of \eqref{eq:new algo}, which follows a similar approach to the one used in \cite{condat2013primal}. Namely, we show that the PDTR method~\eqref{eq:new algo} can be understood as an instance of the forward-reflected-backward method in a different metric. To this end, consider the block-matrix given by
\begin{equation*}\label{eq:M}
 M= \begin{bmatrix}
         \frac{1}{\tau}\Id & -K^\adj \\ -K & \frac{1}{\sigma}\Id
        \end{bmatrix},
\end{equation*}        
which is positive definite when $\tau\sigma\|K\|^2<1$. Let $\langle\cdot,\cdot\rangle_M$ and $\|\cdot\|_M$ respectively denote the inner-product and corresponding norm induced by $M$. Further, let $G\colon\Hilbert\times\Hilbert'\setto\Hilbert\times\Hilbert'$ and $F\colon\Hilbert\times\Hilbert'\to\Hilbert\times\Hilbert'$ be the operators given by
\begin{equation*}\label{eq:GF}
 G=\begin{bmatrix}
      A  & K^\adj \\
      -K & C^{-1} \\
     \end{bmatrix},\qquad
  F=\begin{bmatrix}
      B  & 0 \\
      0 & 0 \\
     \end{bmatrix}. 
\end{equation*}     
Setting $z^k=(x^k,y^k)$, a straightforward calculation shows that \eqref{eq:new algo} can be rewritten as
\begin{equation}\label{eq:simple_incl}
  z^{k+1} = J_{M^{-1}G}\bigl(z^k-2M^{-1}F(z^k)+M^{-1}F(z^{k-1})\bigr).
\end{equation}
Indeed, using the definition of the resolvent, \eqref{eq:simple_incl} is equivalent to
$$ Mz^{k+1} + Gz^{k+1} \ni Mz^k-2F(z^k)+F(z^{k-1}). $$
Then, using the definitions of $M,G$ and $F$, we see that this is the same as
\begin{equation*}
\left\{\begin{aligned}
 \frac{1}{\tau}x^{k+1}+A(x^{k+1}) &\ni \frac{1}{\tau}x^k-K^*y^k-2B(x^k)+B(x^{k-1})  \\
\frac{1}{\sigma}y^{k+1} -2Kx^{k+1}+C^{-1}(y^{k+1}) &\ni -Kx^k+\frac{1}{\sigma}y^k,
\end{aligned}\right.
\end{equation*}
which is equivalent to \eqref{eq:new algo}. To summarise, we have shown that
\eqref{eq:new algo} can be viewed as the forward-reflected-backward
algorithm~\cite{malitsky2020forward} with unit stepsize applied to $0\in
M^{-1}G(z) + M^{-1}F(z)$. Thus, to deduce convergence of $(z^k)$, we need only
to verify that the assumptions of FoRB hold in our setting, to which the rest of this section is dedicated.

\begin{lemma}\label{l:GF}
Suppose $\tau,\sigma>0$ satisfy $\tau\sigma\|K\|^2<1$.  Then the following assertions hold.
\begin{enumerate}[label=(\alph*)]
\item $\binom{x}{y}\in\Hilbert\times\Hilbert'$ satisfies the primal-dual inclusion~\eqref{eq:pd incl} if and only if $\binom{0}{0}\in (M^{-1}F+M^{-1}G)\binom{x}{y}$.
\item $G$ is (maximally) monotone w.r.t.\ $\langle\cdot,\cdot\rangle$ if and only if $M^{-1}G$ is (maximally) monotone w.r.t.\ $\langle\cdot,\cdot\rangle_M$.
\item If $F$ is $L$-Lipschitz continuous w.r.t.\ $\|\cdot\|$, then $M^{-1}F$ is $L_M$-Lipschitz continuous w.r.t.\ $\|\cdot\|_M$ where 
 $$ L_M= \frac{\tau L}{1 - \tau\sigma \|K\|^2}. $$
\end{enumerate}
\end{lemma}
\begin{proof}
(a)~\&~(b)~Follow from the definitions.
(c)~For any $z=(x,y),\bar{z}=(\bar{x},\bar{y})\in\Hilbert\times\Hilbert'$, we have
\begin{equation}\label{eq:lip1}
  \begin{aligned}
\|M^{-1}F(z)-M^{-1}F(\bar{z})\|_M^2
&= \langle F(z)-F(\bar{z}),M^{-1}\bigl(F(z)-F(\bar{z})\bigr)\rangle \\
&= \t \langle B(x)-B(\bar{x}),\bigl(I- \tau\sigma K^\adj K\bigr)^{-1}\bigl(B(x)-B(\bar{x})\bigr)\rangle \\
&\leq \tau L^2\bigl(1-\tau \sigma\|K\|^2\bigr)^{-1}\|x-\bar{x}\|^2,
  \end{aligned}\end{equation}
where the second equality follows from the definition of $F$ and the standard formula for the inverse of a $2\times 2$ block matrix~\cite[Chapter~0.7.3]{horn2012matrix}. Next, we observe that
\begin{equation}\label{eq:lip2}
\begin{aligned}
\n{z-\z}_M^2
&= \frac{1}{\t}\n{x-\x}^2-2\lr{K(x-\x),y-\y}+\frac{1}{\sigma}\n{y-\y}^2\\
&\geq \frac{1}{\tau}\bigl(1-\tau\sigma \n{K}^2\bigr)\n{x-\x}^2.
\end{aligned}
\end{equation}
Thus, altogether, combining \eqref{eq:lip1} and \eqref{eq:lip2} gives
$$ \|M^{-1}F(z)-M^{-1}F(\bar{z})\|_M^2 \leq \frac{\tau^2L^2}{\bigl(1 - \tau\sigma\|K\|^2\bigr)^{2}}\|z-\bar{z}\|_M^2, $$
which establishes the result.
\end{proof}

The following theorem is our main result concerning convergence of \eqref{eq:new algo}.
\begin{theorem}\label{th:main1}
Suppose the primal inclusion~\eqref{eq:mono incl} has at least one solution,
and let $\tau, \sigma>0$ be such that $2\tau L+\tau\sigma\|K\|^2<1$. Then the sequence $(x^k,y^k)$ given by \eqref{eq:new algo} converges to a solution of primal-dual inclusion~\eqref{eq:pd incl}.
\end{theorem}
\begin{proof}   
As discussed before Lemma~\ref{l:GF}, the sequence $z^k=(x^k,y^k)$ satisfies \eqref{eq:simple_incl}
and hence, it can be considered as the sequence generated by the
forward-reflected-backward method~\cite{malitsky2020forward} in the space $\Hilbert\times\Hilbert'$ with
inner-product $\langle\cdot,\cdot\rangle_M$ applied to the inclusion $0\in
\bigl(M^{-1}G + M^{-1}F\bigr)(z)$,  with unit stepsize. Note that, thanks to
Lemma~\ref{l:GF}, $M^{-1}G$ is maximally monotone and  $M^{-1}F$ is monotone and $L_M$-Lipschitz with $L_M=\frac{\tau L}{1 - \tau\sigma \|K\|^2}$. Since $2\tau L+\tau\sigma\|K\|^2 < 1$ is equivalent to $1 < 1/(2L_M)$, \cite[Theorem~2.5]{malitsky2020forward} implies that $(z^k)$ converges to a point $z=(x,y)\in(M^{-1}G+M^{-1}F)^{-1}(0)=(G+F)^{-1}(0)$. This completes the proof.
\end{proof}

\begin{remark}\label{r:pdtr rate}
In the context of Theorem~\ref{th:main1}, the sequence $(z^k)$ satisfies
\begin{equation}\label{eq:rate}
\frac{1}{k}\sum_{i=1}^k\|z^{i+1}-z^i\|^2_M=O\left(\frac{1}{k}\right),\quad \min_{i\leq k}\|z^{i+1}-z^i\|^2_M=o\left(\frac{1}{k}\right).
\end{equation}
To see this, note that according to \cite[Lemma~2.4]{malitsky2020forward}, we have
$$\varphi_{k+1} + \varepsilon\|z^{k+1}-z^k\|^2_M \leq \varphi_k\quad\forall k\in\mathbb{N}, $$
where 
$\varphi_k := \|z^k-z\|^2_M+2\langle B(z^k)-B(z^{k-1}),z-z^k\rangle_M + \frac{1}{2}\|z^k-z^{k-1}\|^2_M \geq 0.$
From this, we deduce that $z^{k+1}-z^k\to 0$ and $\varepsilon\sum_{i=1}^k\|z^{i+1}-z^i\|^2_M\leq \varphi_0$. The latter implies \eqref{eq:rate}.
\end{remark}

\section{An Algorithm for decentralised monotone inclusions}\label{s:dist mono}
Having derived a new primal-dual algorithm, we turn our attention to a connected network of $n$ agents working cooperatively to solve the monotone inclusion
\begin{equation}\label{eq:n inlc}
 0\in\sum_{i=1}^n(A_i+B_i)(x)\subset\Hilbert,
\end{equation}
where $A_1,\dots,A_n\colon\Hilbert\setto\Hilbert$ are maximally monotone, and
$B_1,\dots,B_n\colon\Hilbert\to\Hilbert$ are monotone and $L$-Lipschitz continuous. 

Our strategy for solving \eqref{eq:n inlc} is to reformulate it in the form of
\eqref{eq:mono incl}, followed by applying the PDTR algorithm from the
previous section. To represent~\eqref{eq:n inlc} in the form of \eqref{eq:mono incl}, we shall require an appropriate product space reformulation, which we study systematically in the following section.

\subsection{Product Space Reformulations}
Let $F_1,\dots,F_n\colon\Hilbert\setto\Hilbert$ be set-valued operators, and consider the $n$-operator inclusion
\begin{align}
  0 &\in \sum_{i=1}^nF_i(x)\subset\Hilbert.\label{eq:inc1}
\end{align}
Denote $\bx = (x_1,\dots,x_n)\in \Hilbert^n$ and
$\bF(\bx) = F_1(x_1)\times\dots\times F_n(x_n)$. The standard product approach (see, for example, \cite[Proposition~26.4]{bauschke2017convex}) involves forming the two-operator inclusion 
\begin{equation}\label{eq:inc3}
 0 \in \bF(\bx) + N_{\mathcal{D}}(\bx) \subset \Hilbert^n.
\end{equation}
Although this formulation can be useful for solving \eqref{eq:inc1}, the resulting algorithms do not lend themselves to decentralised implementation over arbitrary networks. Instead, we will need to consider an alternative product space representation. To this end, let $K\colon\Hilbert^n\to\Hilbert'$ be a linear operator with $\ker K=\mathcal{D}$. Consider the inclusion
\begin{align}
  0 &\in \bF(\bx)+K^\adj N_{\{0\}}(K\bx) \subset \Hilbert^n. \label{eq:inc2}
\end{align}      
The following proposition shows that \eqref{eq:inc1} and \eqref{eq:inc2} are equivalent.
\begin{proposition}\label{prop:prod reform}
Let $F_1,\dots,F_n\colon\Hilbert\setto\Hilbert$ be set-valued operators, and let $K\colon\Hilbert^n\to\Hilbert'$ be a linear operator with $\ker K=\mathcal{D}$. Then $\bx=(x_1,\dots,x_n)\in\Hilbert^n$ is a solution of \eqref{eq:inc2} if and only
if $x_1=\dots=x_n$ and $x\coloneqq x_1$ is a solution of \eqref{eq:inc1}.
\end{proposition}
\begin{proof}
Suppose $\bx=(x_1,\dots,x_n)\in\Hilbert^n$ satisfies~\eqref{eq:inc2}. Then, in
particular, $N_{\{0\}}(K\bx)\neq\emptyset$ and hence $K\bx=0$. Since $\ker
K=\mathcal{D}$, it follows that $x \coloneqq x_1=\dots=x_n$. Since
\begin{align*}
0 \in \bF(\bx)+K^\adj N_{\{0\}}(K\bx) 
= \bF(\bx)+\range K^\adj 
= \bF(\bx) +
(\ker K)^\perp = \bF(\bx) +
\mathcal{D}^\perp ,
\end{align*} 
it follows that $0\in\sum_{i=1}^nF_i(x)$. In other words, $x$ satisfies \eqref{eq:inc1} as claimed.

Conversely, suppose $x\in\Hilbert$ satisfies \eqref{eq:inc1} and denote
$\bx\coloneqq (x,\dots,x)\in\Hilbert^n$. Then there exists $\bv=(v_1,\dots,v_n)\in\Hilbert^n$ with
$v_i\in F_i(x)$ such that $\sum_{i=1}^nv_i=0$. That is,
$\bv\in \mathcal{D}^\perp =(\ker K)^\perp=\range K^\adj$ and hence there exists $\by\in\Hilbert'$
such that $\bv=K^\adj\by$. Noting that $N_{\{0\}}(K\bx)=N_{\{0\}}(0)=\Hilbert'$, we therefore have
$$ 0\in\bF(\bx)- \bv =\bF(\bx) -K^\adj\by \subset \bF(\bx)+K^\adj N_{\{0\}}(K\bx).$$
In other words, $\bx$ satisfies~\eqref{eq:inc2}, which completes the proof.
\end{proof}

\begin{remark}
The proposed product space formulation~\eqref{eq:inc2} and the standard product reformulation \eqref{eq:inc3} are equivalent. This can be seen by noting that 
\[ N_{\mathcal{D}}(\bx) =
\begin{cases}
  \mathcal{D}^\perp   & \text{if\, }\bx\in \mathcal{D}\\
  \emptyset & \text{otherwise},
\end{cases}\qquad
K^\adj N_{\{0\}}(K\bx) =
\begin{cases}
   \range K^\adj & \text{if\, } \bx\in\ker K \\
  \emptyset  & \text{otherwise},
\end{cases}
\]
and $\range K^\adj = \mathcal{D}^\perp$. However, from an implementation point of view, the role of these formulations are quite different.
\end{remark}

\subsection{Derivation of the Algorithm}\label{s:derivation}
In this section, it will be convenient to consider linear operators that are not necessarily directly represented as matrices. To this end, let $W\colon\Hilbert\to\Hilbert$ be a linear operator satisfying the following three assumptions
\begin{enumerate}[label=\textnormal{(A\arabic*)}]
\item \textsc{(Self-adjoint property)} $W=W^*$. \label{p:selfadjoint}
\item \textsc{(Kernel property)} $\ker(\Id-W)=\mathcal{D}$. \label{p:kernel}
\item \textsc{(Spectral property)} $\Id\succeq W\succ-\Id$. \label{p:spectral}
\end{enumerate}
In particular, if $W$ is a mixing matrix, then Assumptions~\ref{p:selfadjoint}--\ref{p:spectral} hold with $\Hilbert=\mathbbm{R}^{n\times n}$.

Now, set $K=K^*=(\frac{\Id-W}{2})^{\sfrac 12 }$ which is well-defined due to Assumptions~\ref{p:selfadjoint}~\&~\ref{p:spectral}. Further, denote
\begin{align*}
 \bA(\bx)=A_1(x_1)\times\dots\times A_n(x_n),\qquad 
 \bB(\bx)=(B_1(x_1),\dots,B_n(x_n)). 
\end{align*}
By setting $F_i=A_i+B_i$ in Proposition~\ref{prop:prod reform} and noting that Assumption~\ref{p:kernel} holds, we deduce that the
 inclusion~\eqref{eq:n inlc} is equivalent to
\begin{equation}\label{eq:th ext incl}
  0 \in \bA(\bx)+\bB(\bx) + KN_{\{0\}}(K\bx)\subset\Hilbert^n.
\end{equation}
This inclusion is of the form \eqref{eq:mono incl} and hence we can apply the
PDTR algorithm~\eqref{eq:new algo}. Specifically,  let $\bx^0=\bx^{-1}$,
$\by^0=0$,  $\sigma=\tau^{-1}$ and $C=N_{\{0\}}$. Then \eqref{eq:new algo} becomes
\begin{equation}\label{eq:th new dist alg}
\left\{\begin{aligned}
 \bx^{k+1} &= J_{\tau\bA}\bigl(\bx^k-\tau K\by^k-\tau\bv^k\bigr) \\
 \by^{k+1} &= \by^k+\tau^{-1}K(2\bx^{k+1}-\bx^k)
\end{aligned}\right.
\end{equation}
where $\bv^k\coloneqq 2\bB(\bx^k)-\bB(\bx^{k-1})$. Next, define a new sequence $(\bu^k)$ according to 
$$ \bu^{k+1} = \bx^k-\tau K\by^k-\tau\bv^k \quad\forall k\geq 0. $$
Then $\bu^1 = \bx^0 - \tau\bB(\bx^0)$ and, for $k\geq 1$, we have
\begin{equation*}
\begin{aligned}
 \bu^{k+1} -\bu^k + \tau(\bv^k-\bv^{k-1})
  &= \bx^k  -\bx^{k-1} -\tau K(\by^k-\by^{k-1})   \\
  &= \bx^k  -\bx^{k-1} -\frac{1}{2}(\Id-W)(2\bx^k-\bx^{k-1})\\
  &= W\bx^k-\frac{1}{2}(\Id+W)\bx^{k-1}.
\end{aligned}
\end{equation*}
Thus, altogether, the iteration \eqref{eq:th new dist alg} can be rewritten as
\begin{equation}\label{eq:mono dist alg}
\left\{\begin{aligned}
 \bu^{1} &= \bx^0  - \tau\bB(\bx^0) \\
 \bx^{1} &= J_{\tau\bA}\bigl(\bu^{1}\bigr) \\ 
 \bu^{k+1} &= W\bx^k+\bu^k-\frac{1}{2}(\Id+W)\bx^{k-1}-\tau(\bv^k-\bv^{k-1}) \\
 \bx^{k+1} &= J_{\tau\bA}\bigl(\bu^{k+1}\bigr).
\end{aligned}\right.
\end{equation}
The following theorem provides our convergence analysis of \eqref{eq:mono dist alg}.
\begin{theorem}\label{th:dist0}
Suppose the inclusion \eqref{eq:n inlc} has at least one solution. Let $W$ be a linear operator satisfying Assumptions~\ref{p:selfadjoint}--\ref{p:spectral} and suppose $0<\tau<\frac{1+\lambda_{\min}(W)}{4L}$. Then the sequence $(\bx^k)$ given by~\eqref{eq:mono dist alg} converges to a point $\bx=(x,\dots,x)\in\Hilbert^n$ such that $x$ is solution of \eqref{eq:n inlc}.
\end{theorem}
\begin{proof}
 By appealing to Theorem~\ref{th:main1}, we deduce that $(\bx^k)$ converges to a solution of~\eqref{eq:th ext incl}, provided that $2\tau L+\|K\|^2<1$. Since
\begin{align*} 
\|K\|^2 =  \|K^2\| = \frac{1}{2}\|\Id-W\|
=\frac{1}{2}\lambda_{\max}(\Id-W) 
=\frac{1}{2}(1-\lambda_{\min}(W))>0, 
\end{align*}
 where the last inequality follows from Assumption~\ref{p:spectral}, the stepsize condition is equivalent to requiring that
 $\tau < \frac{1+\lambda_{\min}(W)}{4L}$. The latter holds by assumption and so the proof is complete.
\end{proof}

\begin{remark}
Let $(\bx^k,\by^k)$ be the sequence given by \eqref{eq:th new dist alg}. In the setting of Theorem~\ref{th:dist0}, the optimality residuals satisfy
\begin{align*}
\frac{1}{k}\sum_{i=1}^k\|(\bar{A}+B)(\bx^{k}) + K\by^k\| &=O\left(\frac{1}{k}\right),& \frac{1}{k}\sum_{i=1}^k\|K\bx^{k}\|^2&=O\left(\frac{1}{k}\right), \\
 \min_{i\leq k}\|(\bar{A}+B)(\bx^{k}) + K\by^k\| &=o\left(\frac{1}{k}\right),& \min_{i\leq k}\|K\bx^{k}\|^2&=o\left(\frac{1}{k}\right)
\end{align*}
for a selection $\bar{A}(x^{k})\in A(x^k)$. To see this, apply Remark~\ref{r:pdtr rate} to \eqref{eq:mono dist alg} and note that \eqref{eq:mono dist alg} implies
\begin{multline*}
 (A+B)(\bx^{k+1}) + K\by^{k+1} \ni  \frac{\bx^k-\bx^{k+1}}{\tau} + K(\by^{k+1}-\by^k) +  B(\bx^{k+1})-2B(\bx^k)-B(\bx^{k-1}) 
\end{multline*}
and
 $$ K\bx^{k+1} = \frac{\by^{k+1}-\by^k}{\tau} - K(\bx^{k+1}-\bx^k). $$

\end{remark}

When stated in terms of an abstract operator $W$, the proposed algorithm is not necessarily suitable for a decentralised implementation (since the ``decentralised property'' need not hold). However, it is suitable when $W$ is a mixing matrix. This important special case is summarised in~Algorithm~\ref{a:dist alg}, whose convergence is assured by Corollary~\ref{cor:dist1}.
\begin{algorithm}[!htb]
\caption{Proposed decentralised algorithm for the monotone inclusion~\eqref{eq:n inlc}.\label{a:dist alg}}
Choose a mixing matrix $W\in\mathbbm{R}^{n\times n}$\;
Choose a stepsize $\tau\in\bigl(0,\frac{1+\lambda_{\min}(W)}{4L}\bigr)$\;
Choose any $\bx^0=\bx^{-1}\in\Hilbert^n$, and set $\bv^0 = \bB(\bx^0)$\;
Initialise  variables according to:
     $$ \left\{\begin{aligned}
         \bu^{1} = \bx^0 - \tau \bv^0 \\
         \bx^{1} = J_{\tau\bA}\bigl(\bu^{1}\bigr) \\
        \end{aligned}\right. $$
\For{$k=1,2,\dots,$}{
Update variables according to:
$$ \left\{\begin{aligned}
        \bv^{k\phantom{+1}} &= 2\bB(\bx^{k})-\bB(\bx^{k-1})\\
        \bu^{k+1} &= W\bx^k+\bu^k-\frac{1}{2}(\Id+W)\bx^{k-1} -\tau(\bv^k-\bv^{k-1}) \\
        \bx^{k+1} &= J_{\tau\bA}\bigl(\bu^{k+1}\bigr)
   \end{aligned}\right. $$   
      }      
\end{algorithm}

In what follows, we also provide the equivalent version of Algorithm~\ref{a:dist alg} written in terms of the agents' local variables. Given an arbitrary starting point $\bx^0=(x_1^0,\dots,x_n^0)^\top\in\mathbbm{R}^{n\times p}$, the first iteration, for all agents $i\in\{1,\dots,n\}$, is computed as
\begin{equation*}\label{eq:agent alg 1}
\left\{\begin{aligned}
u^1_i &= x^0_i - \tau B_i(x^0_i) \\
x^1_i &= J_{\tau A_i}(u^1_i) \\
\end{aligned}\right.
\end{equation*}
Then, for all agents $i\in\{1,\dots,n\}$, all subsequent iterations (\emph{i.e.,} $k\geq 1$) are computed according to
\begin{equation*}\label{eq:agent alg 2}
\left\{\begin{aligned}
 v_i^k &= 2B_i(x^k_i)-B_i(x^{k-1}_i) \\
 u^{k+1}_i &= \sum_{i=1}^nw_{ij}x^k_j+u^k_i-\frac 12\bigl(x_{i}^{k-1}+\sum_{j=1}^nw_{ij}x^{k-1}_j\bigr) -\tau(v^k_i-v^{k-1}_i) \\
 x^{k+1}_i &= J_{\tau A_i}\bigl(u^{k+1}_i\bigr).
\end{aligned}\right.
\end{equation*}
We formally state the convergence proof of Algorithm~\ref{a:dist alg}.
\begin{corollary}\label{cor:dist1}
Suppose the inclusion \eqref{eq:n inlc} has at least one solution. Then the sequence $(\bx^k)$ given by Algorithm~\ref{a:dist alg} converges to a point $\bx=(x,\dots,x)\in\Hilbert^n$ such that $x$ is solution of \eqref{eq:n inlc}.
\end{corollary}
\begin{proof}
Since a mixing matrix satisfies Properties~\ref{p:selfadjoint}--\ref{p:spectral}, the result follows from Theorem~\ref{th:dist0}.
\end{proof}

\begin{remark}
In the above derivation, the initialisation $\by^0=0$ was used for the primal-dual algorithm~\eqref{eq:th new dist alg}. This gives rise to the practically convenient initialisation $\bu^1=\bx^0-\tau\bB(\bx^0)$ in Algorithm~\ref{a:dist alg}. However, thanks to Theorem~\ref{th:main1}, convergence of \eqref{eq:th new dist alg} is guaranteed for any initialisation of $\by^0$. For instance, if we were to instead take $\by^0=2\tau^{-1}K\bx^0$, then the resulting initialisation for Algorithm~\ref{a:dist alg} would have been
\begin{align*}
\bu^1 = \bx^0-2K^2\bx^0-\tau\bv^0 &= \bx^0-(I-W)\bx^0-\tau\bv^0 = W\bx^0-\tau\bv^0. 
\end{align*}
The latter initialisation is analogous to the standard initialisation used in PG-EXTRA \cite[Algorithm~1]{shi2015proximal}. However, as noted in \cite[Chapter~11.3.3]{ryu2022large}, it has the disadvantage of requiring a round of communication in the initialisation step.
\end{remark}

\section{An Algorithm for decentralised min-max problems}\label{s:dist minmax}
We now are ready to derive a decentralised algorithm for solving the min-max problem \eqref{eq:minmax2}. We do so, under the assumption that a saddle-point to \eqref{eq:minmax2} exists, and that the following \emph{sum rule} identities hold:
\begin{equation}\label{eq:sum rule}
 \partial\bigl(\sum_{i=1}^nf_i\bigr) = \sum_{i=1}^n\partial f_i,\quad  \partial\bigl(\sum_{i=1}^ng_i\bigr) = \sum_{i=1}^n\partial g_i. 
\end{equation}
Then solving \eqref{eq:minmax2} is equivalent to the monotone inclusion~\eqref{eq:n inlc} with $z=(x,y)\in\Hilbert=\mathbbm{R}^p\times\mathbbm{R}^d$ and
\begin{equation}\label{eq:minmax parts}
A_i=(\partial f_i,\partial g_i),\quad B_i=(\nabla_x\phi_i,-\nabla_y\phi_i). 
\end{equation}

\begin{figure*}[t]
\tikzset{left node/.style={circle,fill=gray!20,draw,minimum size=0.5cm,inner
    sep=2pt}}
\tikzset{right node/.style={circle,fill=gray!70,draw,minimum size=0.5cm,inner sep=2pt}}
\hspace{0.4cm}
\centering
\begin{tikzpicture}

    \node[left node] (1) {$x_1$};
    \node[left node] (2) [below right = 1.2cm and -0.3cm of 1]  {$x_2$};
    \node[left node] (3) [below left = 2.3cm and 1.5cm of 1]  {$x_3$};
    \node[left node] (4) [below right = 3cm and 1.5cm of 1] {$x_4$};
    \path[draw,thick]
    (1) edge node {} (2)
    (2) edge node {} (4)
    (3) edge node {} (2);

  primal variables
\begin{scope}[xshift=11cm]
    \node[right node] (5) {$y_1$};
    \node[right node] (6) [below right = 1.2cm and -0.3cm of 5]  {$y_2$};
    \node[right node] (7) [below left = 2.3cm and 1.5cm of 5]  {$y_3$};
    \node[right node] (8) [below right = 3cm and 1.5cm of 5] {$y_4$};
    \path[draw,thick]
    (5) edge node {} (7)
    (7) edge node {} (8)
    (8) edge node {} (6)
    (5) edge node {} (6);

\end{scope}
    \path[dashed, gray]
    (1) edge node {} (5)
    (2) edge node {} (6)
    (3) edge node {} (7)
    (4) edge node {} (8);

  \end{tikzpicture}
  \caption{{(Left) A connected network of agents for the minimisation variable $\bx$, and (right) a different connected network of agents for the maximisation variable $\by$. Dashed lines depict connections between pairs $(x_i,y_i)$.}}\label{f:diff W}
\label{fig:pd}
\end{figure*}
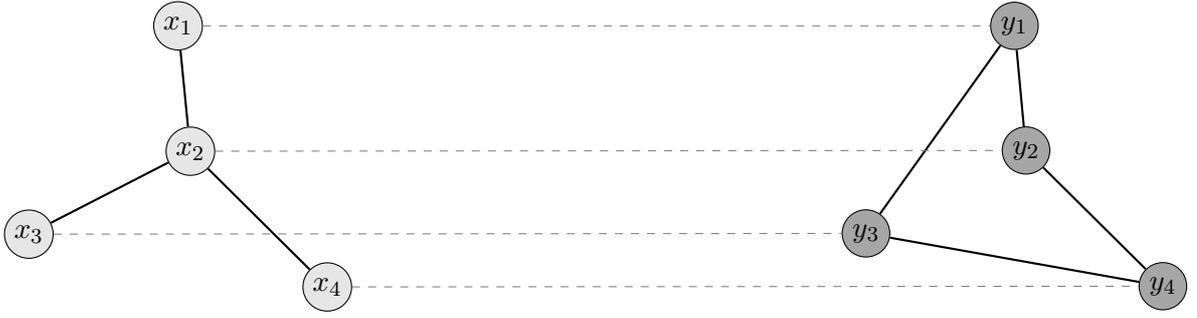

In order to allow different communication rates or even different connected networks with respect to the variables $x$ and $y$, we consider a pair of mixing matrices $W_1,W_2\in\mathbbm{R}^{n\times n}$; see Figure~\ref{f:diff W} for an example. Define a linear operator $W\colon\mathbbm{R}^{n\times p}\times\mathbbm{R}^{n\times d}\to \mathbbm{R}^{n\times p}\times\mathbbm{R}^{n\times d}$ according to
\begin{equation}\label{eq:W}
  W(\bz) = W(\bx,\by) = (W_1\bx,W_2\by).
\end{equation}
Then, applying the iteration \eqref{eq:mono dist alg} with \eqref{eq:minmax parts} and \eqref{eq:W} gives Algorithm~\ref{a:dist minmax}.

\begin{algorithm}[!htb]
\caption{Proposed decentralised algorithm for the min-max problem \eqref{eq:minmax}.}\label{a:dist minmax}
Choose mixing matrices $W_1,W_2\in\mathbbm{R}^{n\times n}$\;
Choose a stepsize $\tau\in\bigl(0,\frac{1+\min\{\lambda_{\min}(W_1),\lambda_{\min}(W_2)\}}{4L}\bigr)$\;
Choose any $\bx^0\in\mathbbm{R}^{n\times p}$ and $\by^0\in\mathbbm{R}^{n\times d}$, and set 
$$\bv^0_x = \nabla_x\phi(\bx^0,\by^0),\qquad \bv^0_y = - \nabla_y\phi(\bx^0,\by^0)$$
Initialise $\bx$ and $\by$ variables according to:
\begin{equation*}
\left\{\begin{aligned}
 \bu^1_x &= \bx^0-\tau\bv^0_x\\
 \bx^{1}   &= \prox_{\tau f}(\bu^{k}_x),
\end{aligned}\right.
\qquad
\left\{\begin{aligned}
  \bu^1_y &= \by^0-\tau\bv^0_y\\
 \by^{1}   &= \prox_{\tau g}(\bu^{k}_y).  
\end{aligned}\right.
\end{equation*}
\For{$k=1,2,\dots,$}{
Update $\bx$-variables according to:
\begin{equation*}\left\{\begin{aligned}
 \bv^k_x &= 2\nabla_x\phi(\bx^k,\by^k)-\nabla_x\phi(\bx^{k-1},\by^{k-1}) \\
 \bu^{k+1}_x &= W_1\bx^k+\bu^k_x-\frac{1}{2}(\Id+W_1)\bx^{k-1}  -\tau(\bv^k_x-\bv^{k-1}_x) \\
 \bx^{k+1}   &= \prox_{\tau f}(\bu^{k+1}_x)
\end{aligned}\right.
\end{equation*}
Update $\by$-variables according to:
\begin{equation*}\left\{\begin{aligned}
 \bv^k_y &= -2\nabla_y\phi(\bx^k,\by^k)+\nabla_y\phi(\bx^{k-1},\by^{k-1}) \\
 \bu^{k+1}_y &= W_2\by^k+\bu^k_y-\frac{1}{2}(\Id+W_2)\by^{k-1} -\tau(\bv^k_y-\bv^{k-1}_y) \\
 \by^{k+1}   &= \prox_{\tau g}(\bu^{k+1}_y).
\end{aligned}\right.
\end{equation*}
}
\end{algorithm}
The following is our main result regarding convergence of Algorithm~\ref{a:dist minmax}.

\begin{theorem}\label{th:dist minmax}
Suppose the min-max problem \eqref{eq:minmax2} has a saddle-point and the sum rules~\eqref{eq:sum rule} hold. Then the sequences $(\bx^k)$ and $(\by^k)$ given by Algorithm~\ref{a:dist minmax} converge, respectively, to points $\bx=(x,\dots,x)^\top\in\mathbbm{R}^{n\times p}$ and $\by=(y,\dots,y)^\top\in\mathbbm{R}^{n\times d}$ such that $(x,y)$ is a solution of~\eqref{eq:minmax2}.
\end{theorem}
\begin{proof}
As explained above, the min-max problem~\eqref{eq:minmax2} is equivalent to \eqref{eq:inc1} with $A_i$ and $B_i$ as in \eqref{eq:minmax parts}. The operator $W$ given in \eqref{eq:W} satisfies Assumptions~\ref{p:selfadjoint}--\ref{p:spectral} where, in verifying \ref{p:kernel}, we note that
$ \bz=(\bx,\by)\in\mathcal{D}
   \iff z_1=\dots=z_n 
   \iff x_1=\dots=x_n\text{~and~}y_1=\dots=y_n 
   \iff \bz=(\bx,\by)\in\ker(\Id-W_1)\times\ker(\Id-W_2)=\ker(\Id-W)$. 
Moreover, we also have that $\lambda_{\min}(W) = \min\{\lambda_{\min}(W_1),\lambda_{\min}(W_2)\}$. Thus, altogether, applying Theorem~\ref{th:dist0} gives that the sequence $(\bz^k)=(\bx^k,\by^k)$ given by
\begin{equation*}
\left\{\begin{aligned}
 \bu^{1} &= \bz^0  - \tau\bB(\bz^0) \\
 \bz^{1} &= J_{\tau\bA}\bigl(\bu^{1}\bigr) \\ 
 \bu^{k+1} &= W\bz^k+\bu^k-\frac{1}{2}(\Id+W)\bz^{k-1}-\tau(\bv^k-\bv^{k-1}) \\
 \bz^{k+1} &= J_{\tau\bA}\bigl(\bu^{k+1}\bigr),
\end{aligned}\right.
\end{equation*}
converges to a solution of \eqref{eq:minmax2}. By extracting the variables $(\bx^k)$ and $(\by^k)$ from $(\bz^k)$, this iteration can written in the form of Algorithm~\ref{a:dist minmax}.
\end{proof}
\begin{remark} 
We note that, in the analysis of PG-EXTRA~\cite{shi2015proximal}, the non-smooth components of the objective function are assumed to have full domain. In this case, the subdifferential sum-rule automatically holds and so its importance is not discussed in \cite{shi2015proximal}. Since we do not assume that $f_1,\dots,f_n,g_1,\dots,g_n$ have full domain, it is necessary for us to include \eqref{eq:sum rule} as an assumption.
\end{remark}

\section{Conclusions and Outlook}
In this work, we have developed and analysed a new simple decentralised
distributed algorithm for solving min-max problems. The algorithm, which can be
considered as the counterpart of PG-EXTRA for solving min-max problems, has many
desirable convergence properties including non-decaying stepsizes, minimal
communication requirements, and an iteration tailored specifically to the
properties of the problem. The convergence analysis of the method is based on a new primal-dual algorithm, known as the primal-dual twice referenced algorithm (PDTR). Further study of PDTR algorithm may be of interest in its own right.

\section*{Acknowledgements}
We thank the referees for their constructive comments which
have helped improved the manuscript. MKT is supported in part by Australian Research Council grants DE200100063 and DP230101749.

\bibliographystyle{abbrv}
\bibliography{biblio}

\end{document}